\documentclass{amsart}
\usepackage{epsfig}
\usepackage{amsfonts,amssymb,amscd,amsmath,euscript,enumerate,verbatim,calc}
\theoremstyle{plain}

\newtheorem{theorem}{Theorem}[section]
\newtheorem{lemma}[theorem]{Lemma}
\newtheorem{proposition}[theorem]{Proposition}

\newtheorem{conjecture}[theorem]{Conjecture}
\newtheorem{corollary}[theorem]{Corollary}
\theoremstyle{definition}

\def\Fq{{\mathbb F}_q}
\def\FF{{\mathbb F}}
\def\PP{{\mathbb P}}
\def\Fqm{{\mathbb F}_{q^m}}
\def\Fqmn{{\mathbb F}_{q^{mn}}}
\def\Fqtn{{\mathbb F}_{q^{2n}}}

\newcommand{\GL}{\operatorname{GL}}
\newcommand{\T}{\operatorname{T}}
\newcommand{\TGL}{\operatorname{TGL}}
\newcommand{\M}{\operatorname{M}}

\newcommand{\Cf}{{\sf C}_f}
\newcommand{\Cg}{{\sf C}_g}

\newcommand{\bcms}{{\rm BCMS}(m,n;q)}
\newcommand{\bcmi}{{\rm BCMI}(m,n;q)}
\newcommand{\bcmst}{{\rm BCMS}(2,n;q)}
\newcommand{\bcmit}{{\rm BCMI}(2,n;q)}

\newcommand{\pmn}{{\EuScript P}(mn;q)}
\newcommand{\imn}{{\EuScript I}(mn;q)}
\newcommand{\itn}{{\EuScript I}(2n;q)}
\newcommand{\ptn}{{\EuScript P}(2n;q)}

\newcommand{\bav}{{\EuScript B}^{\alpha}_{(v_1,v_2)}}
\newcommand{\baw}{{\EuScript B}^{\alpha}_{(w_1,w_2)}}
\newcommand{\bavm}{{\EuScript B}^{\alpha}_{(v_1,\dots, v_m)}}

\newcommand{\B}{{\EuScript B}}
\newcommand{\Sb}{{\EuScript S}_{\beta}}
\begin{document}
\title[Block Companion Singer Cycles and Coprime Polynomial Pairs]{Block Companion Singer Cycles, 
Primitive Recursive Vector Sequences, and Coprime Polynomial Pairs  over Finite Fields} 
\author{Sudhir R. Ghorpade}
\address{Department of Mathematics, 
Indian Institute of Technology Bombay,\newline \indent
Powai, Mumbai 400076, India.}
\email{srg@math.iitb.ac.in}

\author{Samrith Ram}
\address{Department of Mathematics,
Indian Institute of Technology Bombay,\newline \indent
Powai, Mumbai 400076, India.}
\email{samrithram@gmail.com}



\begin{abstract}
We discuss a conjecture concerning the enumeration of nonsingular matrices over a finite field that are block companion and whose order is the maximum possible in the corresponding general linear group. A special case is proved using some recent results on the probability that a pair of polynomials with coefficients in a finite field is coprime. Connection with an older problem of Niederreiter about the number of splitting subspaces of a given dimension are outlined and an asymptotic version of the conjectural formula is established. Some applications to the enumeration of nonsingular Toeplitz matrices of a given size over a finite field are also discussed.  
\end{abstract}
\date{\today}
\maketitle
\section{Introduction}

Let $\Fq$ denote the finite field with $q$ elements and let $m,n$ be positive integers.
For any positive integer $d$, we denote by $\M_d(\Fq)$ 
the set of all $d\times d$ matrices with entries in $\Fq$, and by $\GL_d(\Fq)$ 
the group of all 
nonsingular matrices in $\M_d(\Fq)$.   
By an $(m,n)$-\emph{block companion matrix} over $\Fq$ we mean $T\in \M_{mn}(\Fq)$ of the form
\begin{equation}
\label{typeT} 
T =
\begin {pmatrix}
\mathbf{0} & \mathbf{0} & \mathbf{0} & . & . & \mathbf{0} & \mathbf{0} & C_0\\
I_m & \mathbf{0} & \mathbf{0} & . & . & \mathbf{0} & \mathbf{0} & C_1\\
. & . & . & . & . & . & . & .\\
. & . & . & . & . & . & . & .\\
\mathbf{0} & \mathbf{0} & \mathbf{0} & . & . & I_m & \mathbf{0} & C_{n-2}\\
\mathbf{0} & \mathbf{0} & \mathbf{0} & . & . & \mathbf{0} & I_m & C_{n-1}
\end {pmatrix}, 
\end{equation}
where $C_0, C_1, \dots , C_{n-1}\in \M_m(\Fq)$ and $I_m$ denotes the $m\times m$ identity matrix over $\Fq$, while $\mathbf{0}$ indicates the zero matrix in $\M_m(\Fq)$. If such a matrix $T$ is a Singer cycle in $\GL_{mn}(\Fq)$, that is, if $T$ is nonsingular and the order of $T$ in the 
group $\GL_{mn}(\Fq)$ is the maximum possible (viz., $q^{mn}-1$), then we will 
call it 
a \emph{$(m,n)$-block companion Singer cycle} over $\Fq$. We are primarily interested in the following. 

\begin{conjecture}
\label{conj0}
The number of $(m,n)$-block companion Singer cycles over $\Fq$ is 
\begin{equation}
\label{NoSigmaLFSR}
\frac{\phi(q^{mn}-1)}{mn} \, q^{m(m-1)(n-1)} \displaystyle \prod_{i=1}^{m-1}(q^m-q^i),
\end{equation}
where $\phi$ is the Euler totient function.
\end{conjecture}

This conjecture arose in the study by Zeng, Han and He \cite{Zeng} 
of word-oriented linear feedback shift registers, called $\sigma$-LFSRs and is 
equivalent to showing that the number of primitive $\sigma$-LFSRs of order $n$ over $\Fqm$ is given by \eqref{NoSigmaLFSR} above. It may be noted 
that a special case 
of $\sigma$-LFSRs appears earlier in the work of Tsaban and Vishne \cite{TV}. Moreover, the $\sigma$-LFSRs turn out to be essentially the same as recursive vector sequences studied by Niederreiter \cite{N1,N2} in the context of his work on pseudorandom number generation and his multiple-recursive matrix method. As such the question about the enumeration of block companion Singer cycles over $\Fq$ is intimately related to the open problem about the determination of the total number of $\sigma$-splitting subspaces over $\Fq$ of a given dimension. (See Section \ref{sec:splitsubsp} for details.) Nonetheless, the explicit conjectural formula \eqref{NoSigmaLFSR} should be attributed to Zeng, Han and He \cite{Zeng}, at least in the binary case, whereas the above formulation in the $q$-ary case is as in \cite{GSM}. Although there is significant numerical evidence in its favour, Conjecture \ref{conj0} is open, in general, except in the trivial case $m=1$ (and any $n$) and the not-so-trivial special case $n=1$ (and any $m$), where it is proved in \cite{GSM}. A plausible approach to proving Conjecture \ref{conj0} in the general case was proposed in \cite{GSM} and a more refined, but perhaps more amenable, conjecture called the Fiber Conjecture was formulated there. 

In this paper, we prove that the Fiber Conjecture and, as an immediate consequence,  Conjecture \ref{conj0}, holds in the affirmative in the case $m=2$ (and any $n$). In fact, we consider a more general version of the Fiber Conjecture, called Irreducible Fiber Conjecture, and show that it is
valid when $m=2$. One of the key tools used is the recent work on the question of determining the probability of two randomly chosen polynomials of a given positive degree with coefficients in $\Fq$ being relatively prime. This question can be traced back to an exercise in Knuth's book 
\cite[\S 4.6.1, Ex. 5]{K} (see also \cite[Rem. 4.2]{GGR}). More recently, it arose in the study by Corteel, Savage, Wilf, and Zeilberger \cite{CSWZ} of Euler's pentagonal sieve in the theory of partitions and has led to a number of developments; we refer to the subsequent work of Reifegerate \cite{Reif},  Benjamin and Bennett \cite{BB},  Gao and Panario \cite{GP}, Hao and Mullen \cite{HM}, and of Garc\'{i}a-Armas, Ghorpade and Ram \cite{GGR} for more on this topic. 
While the general case of Conjecture \ref{conj0} as well as
Niederreiter's splitting subspace problem still remains open, we provide a quantitative version of the latter together with a refinement, which imply the former. (See Section \ref{sec:splitsubsp} for details). Moreover, in Section \ref{sec:asymp}, we give an asymptotic formula for the cardinality of an  irreducible fiber, which appears to strengthen the validity of the conjectural formula \eqref{NoSigmaLFSR}. Finally, 
as an application of some of the methods used in our proof, we deduce a 
formula for the number of nonsingular Toeplitz matrices (or equivalently, the number of nonsingular Hankel matrices) over $\Fq$, which has also been of some recent interest. 
%

\section{The Characteristic Map} 
\label{charmap}
Denote, as usual, by $\Fq[X]$ the ring of polynomials in one variable $X$ with coefficients in $\Fq$. 
Recall that a polynomial 
in $\Fq[X]$ of degree $d \ge 1$ is said to be \emph{primitive} if it is the minimal polynomial over $\Fq$ of a 
generator of the cyclic group $\mathbb F_{q^d}^*$ of nonzero elements of the finite field $\mathbb F_{q^d}$. 
Fix, throughout this paper, positive integers $m$ and $n$.   
Let 
$$
\pmn := \left\{p(X)\in \Fq[X] : p(X) \mbox{ is primitive of degree } mn\right\}
$$
and let 
$$
\imn := \left\{p(X)\in \Fq[X] : p(X) \mbox{ is monic and irreducible of degree } mn\right\}.
$$
Evidently, $\pmn \subseteq \imn$, but the reverse inclusion is not true, in general. 
The cardinalities of these sets are well known (cf. \cite[\S 2]{GSM}, \cite[p. 93]{LN}); 
namely,
\begin{equation}
\label{cardpmnimn}
\left|\pmn \right| = \frac{\phi(q^{mn}-1)}{mn} \quad \text{and} \quad 
\left|\imn \right| = \frac{1}{mn} \sum_{d\mid mn} \mu\left(\frac{mn}{d}\right)q^d,
\end{equation}
where $\mu$ denotes the M\"obius function. 

The map which associates to an $mn\times mn$ matrix its characteristic polynomial, viz., 
$$
\Phi :\M_{mn}(\Fq) \to \Fq[X] \quad \mbox{ defined by } \quad \Phi(T): = \det\left(XI_{mn} - T \right)
$$
will often be referred to as the \emph{characteristic map}. We denote by $\bcms$ the set of $(m,n)$-block companion Singer cycles over $\Fq$, and by 
$\bcmi$ the set of $(m,n)$-block companion matrices over $\Fq$ having an irreducible characteristic polynomial. Evidently, $\bcms \subseteq \bcmi$ 
and $\Phi$ maps $\bcmi$ into $\imn$. A little less obvious, yet elementary, fact is that a nonsingular matrix is a Singer cycle if and only if its  characteristic polynomial is primitive (see, e.g., \cite[Prop. 3.1]{GSM}); in particular, $\Phi$ maps $\bcms$ into $\pmn$. As a result, restrictions of $\Phi$ yield the following maps:
$$
\Psi: \bcms \to \pmn \quad \mbox{and} \quad \Theta: \bcmi\to \imn.
$$
The following result is proved in \cite[Theorem 6.1]{GSM}.

\begin{proposition}
\label{surjpsi}
$\Psi$ is surjective.
\end{proposition}

Here is a small generalization of Proposition \ref{surjpsi} for which a proof is included. This can also be viewed as an alternative, and slightly shorter, proof of Proposition~\ref{surjpsi} compared to the one given in \cite{GSM}.

\begin{proposition}
\label{surjtheta}
$\Theta$ is surjective and hence so is $\Psi$.
\end{proposition}

\begin{proof}
Let $f\in \imn$. If $\alpha\in \Fqmn$ is a root of $f$, then 
$\Fqmn = \Fq(\alpha)=\Fqm(\alpha)$. 
In particular, $[\Fqm(\alpha):\Fqm] =n$ 
and moreover, if $g\in \Fqm[X]$ denotes the minimal polynomial of $\alpha$ over $\Fqm$, then $\deg g = n$ and $g$ divides $f$ in $\Fqm[X]$. Write 
$g = X^n - \beta_{n-1}X^{n-1} - \dots - \beta_1 X - \beta_0$. Now for any $\beta \in \Fqm$, let $L_{\beta}: \Fqm \to \Fqm$ denote the $\Fq$-linear transformation defined by $L_{\beta}(x) := \beta x$, and let $A_{\beta}\in M_m(\Fq)$ be the matrix of $L_{\beta}$ with respect to a fixed $\Fq$-basis of $\Fqm$. It is clear that for any  $\beta, \gamma\in \Fqm \text{ and } \lambda\in \Fq$, we have
\begin{equation}
\label{Abetagamma}
A_{\beta+\gamma} = A_{\beta}+A_{\gamma}, \quad A_{\beta\gamma} = A_{\beta}A_{\gamma} \quad \text{and} \quad A_{\lambda\beta} = \lambda A_{\beta}.
\end{equation}
Consider the companion matrix $\Cg\in M_n(\Fqm)$ 
of $g$ and the corresponding $(m,n)$-block companion matrix $T \in M_{mn}(\Fq)$, namely,
$$ 
\Cg =
\begin{pmatrix}
0 & 0 &  \dots & 0 & \beta_0\\
1 & 0 & \dots  & 0 & \beta_1\\ 
\vdots &  & \ddots  &   & \vdots \\ 
0 & 0 & \dots  & 1 & \beta_{n-1}
\end{pmatrix}
\quad \text{and} \quad 
T =
\begin{pmatrix}
\mathbf{0} & \mathbf{0}  & \dots  & \mathbf{0} & C_0\\
I_m        & \mathbf{0}  &  \dots & \mathbf{0} & C_1\\
\vdots     &             & \ddots &  .         & \vdots \\
\mathbf{0} & \mathbf{0}  & \dots & I_m         & C_{n-1}
\end{pmatrix}, 
$$
where we have let $C_i= A_{\beta_i}$ for $i=0,1,\dots , n-1$. By the Cayley-Hamilton Theorem, $g(\Cg)=0$ and hence $f(\Cg)=0$. The last equation corresponds to $n^2$ polynomial expressions in $\beta_0, \beta_1, \dots , \beta_{n-1}$ with coefficients in $\Fq$ being equal to zero. In view of \eqref{Abetagamma}, these equations continue to hold if $\beta_i$'s are replaced by $C_i$'s. Consequently, $f(T)=0$. Since $f\in \Fq[X]$ is monic and irreducible of degree $mn$, it follows that $f$ is the characteristic polynomial of $T$, i.e., $f = \Theta (T)$.  
\end{proof}

As an immediate consequence of Proposition \ref{surjtheta}, we obtain natural decompositions of $\bcms$ and 
$\bcmi$ as 
disjoint unions of the fibers of the maps $\Psi$ and $\Theta$, respectively. 
This decomposition of $\bcms$ and 
Proposition \ref{surjpsi} suggested the following refined version proposed in \cite{GSM} of Conjecture \ref{conj0}. 
\begin{conjecture}
\label{FibConj1}
$\left|\Psi^{-1}\left(f\right)\right| = q^{m(m-1)(n-1)} \displaystyle \prod_{i=1}^{m-1}(q^m-q^i)$
for any $f\in \pmn$. 
\end{conjecture}

In light of Proposition \ref{surjtheta}, we propose the following 
more general version of Conjecture \ref{FibConj1}. 

\begin{conjecture}
\label{FibConj2}
$\left|\Theta^{-1}\left(f\right)\right| = q^{m(m-1)(n-1)} \displaystyle \prod_{i=1}^{m-1}(q^m-q^i)$
for any $f\in \imn$.
\end{conjecture}

It is clear that if Conjecture \ref{FibConj2} holds in the affirmative, then so do Conjecture \ref{FibConj1} and Conjecture \ref{conj0}. We may refer to Conjecture \ref{FibConj2} as the 
{\em Irreducible Fiber Conjecture}. Moreover, Conjecture \ref{FibConj1}, which has hitherto been called \emph{Fiber Conjecture},  may now be referred to as the {\em Primitive Fiber Conjecture}. 

\section{Relatively Prime Polynomials} 
\label{sec:relprimepoly}

Let us begin by recalling a result about relatively prime polynomials, namely, 
{\cite[Prop. 3]{CSWZ}} (see also \cite[Exer. 5 of \S 4.6.1]{K} and \cite[Thm. 4.1]{GGR}), 
which was alluded to in the introduction. In this section, $r$ will denote an integer $\ge 2$ and, as before, $n$ is a fixed positive integer. 

\begin{proposition} 
\label{cswz}
The number of coprime $r$-tuples of monic polynomials of
degree $n$ over $\Fq$ is $q^{rn}- q^{r(n-1)+1}$. Alternatively, if $r$ monic polynomials 
in $\Fq[X]$ are chosen independently and uniformly at random, then the probability that 
they are relatively prime is $1-1/q^{r-1}$. 
\end{proposition}

A special case 
of the above result implies that there is a $50\%$ chance that two monic polynomials of a given positive degree in $\FF_2[X]$ are coprime. 
With this in view, Corteel, Savage, Wilf, and Zeilberger \cite{CSWZ} asked for an explicit bijection between coprime and non-coprime 
pairs of monic polynomials of a given positive degree in $\FF_2[X]$. 
A nice answer was given by Benjamin and Bennett 
who proved, more generally, the following result in \cite[Cor. 6]{BB}.
 
\begin{proposition} 
\label{bb}
If $r$ polynomials of degree lees than $n$ in $\Fq[X]$ are randomly chosen, then the probability that they are relatively prime is 
$$
1 - \frac{1}{q^{r-1}} + \frac{q-1}{q^{rn}}.
$$
\end{proposition}

For our purpose, the following consequence 
of the above result will be useful. 

%

\begin{corollary}
\label{corbb}
Let $\Sigma$ denote the set of pairs $(f,g)$ of nonzero polynomials in $\Fq[X]$ of degree $< n$ such that $f$ and $g$ are  relatively prime and moreover $g$ is monic. Then the cardinality of $\Sigma$ is equal to $(q^{2n-1} - 1)$. 
\end{corollary}

\begin{proof}
Since the number of pairs of polynomials of degree $<n$ in $\Fq[X]$ is $q^{2n}$, by Proposition \ref{bb}, the number of coprime pairs 
of polynomials in $\Fq[X]$ of degree $< n$ is equal to $(q^{2n-1}+1)(q-1)$. 
Now, as per the standard conventions, the only polynomials that are coprime to the zero polynomial are the nonzero constant polynomials. 
Hence  if $\Sigma_1$ denotes the set of coprime pairs 
of nonzero polynomials in $\Fq[X]$ of degree $< n$, then 
$\left|\Sigma_1\right|=(q^{2n-1}+1)(q-1) - 2(q-1)=(q^{2n-1}-1)(q-1)$. 
Finally, since $\Sigma = \{(f,g)\in \Sigma_1 : g \text{ is monic}\}$, it follows that $\left|\Sigma\right|=\left|\Sigma_1\right|/(q-1)$.
\end{proof}
\section{The Case $m=2$} 
\label{sec:mtwo} 

%
Given any $\alpha, v_1, v_2\in \Fqtn$, we let 
$$
\bav:=\left\{v_1, \, v_2, \, \alpha v_1, \, \alpha v_2, \, \dots , \,  \alpha^{n-1} v_1,\, \alpha^{n-1} v_2\right\},
$$  
with the proviso that $\bav$ is to be regarded as an ordered set with $2n$ elements; in most applications 
it will be an ordered basis of $\Fqtn$ over $\Fq$. 
Our first step is to relate the fibers of $\Theta$ to ordered bases of the form $\bav$.
 
\begin{lemma}
\label{lalpha}
Let $f\in \itn$ and let $\alpha\in \Fqtn$ be a root of $f$. As before, let $L_{\alpha}: \Fqtn \to \Fqtn$ denote the $\Fq$-linear transformation 
defined by $L_{\alpha}(x) := \alpha x$ for $x\in \Fqtn$, and let $T\in \M_{2n}(\Fq)$. 
Then
$T \in  \Theta^{-1}(f)$ if and only if $T$ is the matrix of $L_{\alpha}$ with respect to an ordered basis of the form $\bav$ for some $v_1, v_2\in \Fqtn$.
\end{lemma}

\begin{proof}
Since $f$ is irreducible, $\left\{1, \alpha,  \dots , \alpha^{2n-1}\right\}$ is an $\Fq$-basis of $\Fqtn$. Moreover, since $f$ is also monic, the matrix of $L_{\alpha}$ with respect to this basis is precisely the companion matrix $\Cf$ of $f$. 

Suppose $T \in  \Theta^{-1}(f)$. Then the monic irreducible polynomial $f$ is the characteristic polynomial of $T$. It follows that $T$ and
$\Cf$ have the same invariant factors and hence they are similar. Consequently, $T$ is the matrix of $L_{\alpha}$ with respect to some ordered 
$\Fq$-basis $\B$ of $\Fqtn$. Further since $T$ is a $(2,n)$-block companion matrix, we see that $\B$ must be of the form $\bav$ for some $v_1, v_2\in \Fqtn$. 

Conversely, suppose $T$ is the matrix of $L_{\alpha}$ with respect to an ordered basis of the form $\bav$ for some $v_1, v_2\in \Fqtn$. Then 
$T$ is clearly a $(2,n)$-block companion matrix and moreover, $T$ is similar to $\Cf$. It follows that $T \in  \Theta^{-1}(f)$.
\end{proof}

The next step is to count the number of ordered bases of the form $\bav$, and this is where the results of the previous section will turn out to be handy. 

\begin{lemma}
\label{nobases}
Fix $f\in \itn$ and a root $\alpha\in \Fqtn$ of $f$. Then the number of ordered bases of the form $\bav$, as $v_1, v_2$ vary over $\Fqtn$, is equal to $q^{2n-1}(q-1)(q^{2n} - 1)$.
\end{lemma}

\begin{proof}
First, fix any $v_1\in \Fqtn$ with $v_1\ne 0$. Then for any $v_2\in \Fqtn$, the ordered set $\bav$ is an $\Fq$-basis of $\Fqtn$ if and only if the ordered set
$$
\Sb := \left\{1, \, \beta, \, \alpha , \, \alpha \beta, \, \dots , \,  \alpha^{n-1} ,\, \alpha^{n-1} \beta\right\}
$$  
is linearly independent over $\Fq$, where $\beta: = v_2/v_1$. Now, $1, \alpha, \dots ,\alpha^{2n-1}$ are linearly independent over $\Fq$ and in particular, so are $1, \alpha, \dots ,\alpha^{n-1}$. Thus for any $\beta \in \Fqtn^*$, the ordered set $\Sb$ is $\Fq$-independent if and only if $\beta$ cannot be expressed as 
$$
\frac{a_0+ a_1 \alpha + \cdots + a_{n-1} \alpha^{n-1}}{b_0+ b_1 \alpha + \cdots + b_{n-1} \alpha^{n-1}}
$$
for some $a_i,b_i\in \Fq$ such that not all $a_i$ are zero and not all $b_i$ are zero ($0\le i\le n-1$). It follows that 
$\big\{\beta \in \Fqtn^* : \Sb \text{ is linearly independent}\big\} =  \Fqtn^* \setminus \Sigma_{\alpha}$, where 
$$
\Sigma_{\alpha} : = \left\{\frac{f(\alpha)}{g(\alpha)} : f,g\in \Fq[X]^*, \ \deg (f) \le n-1, \text{ and } \deg (g) \le n-1\; \right\}.
$$
Now if $\Sigma$ is as in Corollary \ref{corbb}, then the map $\Sigma \to \Sigma_{\alpha}$ given by $(f,g)\mapsto f(\alpha)/g(\alpha)$ is 
clearly well defined and surjective. Moreover, if $(f_1,g_1), (f_2,g_2)\in \Sigma$ are such that $f_1(\alpha)g_2(\alpha)= f_2(\alpha)g_1(\alpha)$, then $f_1(X)g_2(X)=f_2(X)g_1(X)$ because the minimal polynomial of $\alpha$ over $\Fq$ has degree $2n$. Further since $f_i$ and $g_i$ are coprime for $i=1,2$ and since $g_1, g_2$ are monic, it follows that $g_1=g_2$ and therefore $f_1=f_2$. Thus $\Sigma_{\alpha}$ is in bijection with $\Sigma$, 
and hence by Corollary \ref{corbb}, 
$$
\left|\left\{\beta \in \Fqtn^* : \Sb \text{ is linearly independent }\right\}\right| = (q^{2n}-1) - (q^{2n-1}-1) = q^{2n-1}(q-1).
$$
Finally, if we vary $v_1$ over the $(q^{2n}-1)$ elements of $\Fqtn^*$, then we readily see that the number of ordered bases of the form $\bav$ is 
equal to $q^{2n-1}(q-1)(q^{2n}-1)$.
\end{proof}

It is possible that two different bases of the form $\bav$ can give rise to the same matrix. This redundancy can be quantified using the centralizer.

\begin{lemma}
\label{centralizer}
Let $f, \alpha$ and $L_{\alpha}$ be as in Lemma \ref{lalpha}. Then there are exactly $(q^{2n}-1)$ ordered bases of the form $\bav$ such that that the matrix of $L_{\alpha}$ with respect to each of these bases is the same. 
\end{lemma}

\begin{proof}
Suppose $T$ is the matrix of $L_{\alpha}$ with respect to an ordered basis $\bav$  for some $v_1, v_2\in \Fqtn$. 
If $T$ is also the matrix of $L_{\alpha}$ with respect to $\baw$ for some $w_1, w_2\in \Fqtn$, then the 
``change of basis matrix'' that transforms $\bav$ into $\baw$ is a 
$2n \times 2n$ invertible matrix $P$ over $\Fq$ 
with the property that $P^{-1}TP = T$. Conversely if $P\in \GL_{2n}(\Fq)$ is in the centralizer $Z(T)$, that is, if $P^{-1}TP = T$, then $P$ transforms $\bav$ into an ordered basis with respect to which the matrix of $L_{\alpha}$ is $T$ and (therefore) it is necessarily of the form $\baw$ for some $w_1, w_2\in \Fqtn$. It follows that the desired number of ordered bases is $|Z(T)|$. Finally, since the linear transformation $L_{\alpha}$ is cyclic with $f$ as its minimal (as well as characteristic) polynomial, by a 
theorem of Frobenius \cite[Thm. 3.16 and its corollary]{jac}, we see that $Z(T)$ consists only of polynomials in $T$.  
Consequently, $Z(T)\cup \{0\}$ is the $\Fq$-algebra of polynomials in $T$, which is isomorphic to $\Fq[X]/\left\langle f\right\rangle$.
Hence $|Z(T)|=q^{2n}-1$. 
\end{proof}

The following result shows that Conjectures \ref{FibConj2}, \ref{FibConj1}, and \ref{conj0} hold in the affirmative when $m=2$.

\begin{theorem}
\label{mtwoconj}
$\left|\Theta^{-1}\left(f\right)\right| = q^{2n-1} (q-1)$
for any $f\in \itn$.  In particular, $\left|\Psi^{-1}\left(f\right)\right| = q^{2n-1} (q-1)$
for any $f\in \ptn$. Consequently, 
$$
\left|\bcmst \right| = \frac{\phi(q^{2n}-1)}{2n} \, q^{2n-1} (q-1) 
$$
and
$$ 
\left|\bcmit \right| = \frac{1}{2n} \left(\sum_{d\mid 2n} \mu\left(\frac{2n}{d}\right)q^d \right)q^{2n-1} (q-1).
$$
\end{theorem}

\begin{proof}
By Lemmas \ref{lalpha}, \ref{nobases}, and \ref{centralizer}, we readily see that 
$$
\left|\Theta^{-1}\left(f\right)\right| = \frac{q^{2n-1}(q-1)(q^{2n}-1)}{(q^{2n}-1)} = q^{2n-1} (q-1) \quad 
\text{ for any $f\in \itn$.}
$$
Since a nonsingular matrix is a Singer cycle if and only if its characteristic polynomial is primitive \cite[Prop. 3.1]{GSM}, this implies, in particular, that $\left|\Psi^{-1}\left(f\right)\right| = q^{2n-1} (q-1)$
for any $f\in \ptn$. Consequently, we obtain the desired formulae for $\left|\bcmst \right|$ and $\left|\bcmit \right|$ using
\eqref{cardpmnimn} and Proposition~\ref{surjtheta}.
\end{proof}

\section{Splitting Subspaces} 
\label{sec:splitsubsp}

Let $\sigma \in \Fqmn$. Following Niederreiter \cite{N2}, we call an $m$-dimensional $\Fq$-linear subspace $W$ of $\Fqmn$ to be 
\emph{$\sigma$-splitting} if $\Fqmn = W \oplus \sigma W \oplus \cdots \oplus \sigma^{n-1}W$. Define 
$$
S(\sigma, m,n;q) := \text{the number of $\sigma$-splitting subspaces of $\Fqmn$ of dimension $m$}.
$$
Note that for an arbitrary $\sigma\in \Fqmn$, there may not be any $\sigma$-splitting subspace; for example, if $\sigma\in \Fq$, then $\sigma^iW=W$ for every $m$-dimensional subspace $W$ and every $i\ge 0$, and so $W$ cannot be $\sigma$-splitting if $n>1$. But if $n=1$, then   
the only $m$-dimensional subspace, viz., $W=\Fqmn$, is $\sigma$-splitting for every  $\sigma\in \Fqmn$; in particular, $S(\sigma, m,1;q)=1$. 
On the other hand, if $m=1$ and if $\alpha\in \Fqmn=\FF_{q^n}$ is such that $\Fqmn=\Fq(\alpha)$, then every $1$-dimensional subspace is 
$\alpha$-splitting and so $S(\alpha,1,n;q)= (q^n-1)/(q-1)$.  

Determination of $S(\sigma, m,n;q)$, where $\sigma$ is a primitive element of $\Fqmn^*$, is stated as an open problem in \cite[p. 11]{N2} and Professor Niederreiter has informed us that the problem is still open. We shall see below that this problem is essentially equivalent to the Irreducible Fiber Conjecture, and this will allow us to formulate a quantitative version of the problem. 

First, let us observe that some of the notions and results of Section \ref{sec:mtwo} extend readily to the case of arbitrary $m$. 
Given any $\alpha, v_1, \dots, v_m\in \Fqmn$, we let 
$$
\bavm:=\left\{v_1, \dots , v_m, \, \alpha v_1, \dots , \alpha v_m, \, \dots , \,  \alpha^{n-1} v_1,\dots , \alpha^{n-1} v_m\right\}, 
$$  
with the proviso that $\bavm$ is to be regarded as an ordered set with $mn$ elements.
Also, let $L_{\alpha}: \Fqmn \to \Fqmn$ denote the $\Fq$-linear transformation 
defined by $L_{\alpha}(x) := \alpha x$ for $x\in \Fqmn$. Proofs of the following two results are straightforward extensions of the proofs of Lemmas 
\ref{lalpha} and \ref{centralizer} and are left to the reader. 
 
\begin{lemma}
\label{lalpham}
Let $T\in \M_{mn}(\Fq)$, $f\in \imn$ and let $\alpha\in \Fqmn$ be a root of $f$. 
Then
$T \in  \Theta^{-1}(f)$ if and only if $T$ is the matrix of $L_{\alpha}$ with respect to an ordered basis of the form $\bavm$ for some $v_1, \dots v_m\in \Fqmn$.
\end{lemma}

\begin{proof} 
Similar to the proof of Lemma \ref{lalpha}.
\end{proof}

\begin{lemma}
\label{centralizerm}
Let $f\in \imn$ and let $\alpha\in \Fqmn$ be a root of $f$. 
Then there are exactly $(q^{mn}-1)$ ordered bases of the form $\bavm$ such that that the matrix of $L_{\alpha}$ with respect to each of these bases is the same. 
\end{lemma}

\begin{proof}
Similar to the proof of Lemma \ref{centralizer}.
\end{proof}

Determining the number of bases of the form $\bavm$ seems quite difficult, in general, but we can certainly give this a name. Thus, for any $\alpha\in \Fqmn$ such that $\Fqmn = \Fq(\alpha)$, we define
$$
N(\alpha, m,n;q) := \text{the number of ordered bases of $\Fqmn$ of the form $\bavm$}. 
$$
As an immediate consequence of Lemmas \ref{lalpham} and \ref{centralizerm}, we see that 
\begin{equation}
\label{fiberandN}
\left|\Theta^{-1}(f)\right| = \frac{N(\alpha, m,n;q)}{q^{mn}-1} \quad \text{for any $f\in \imn$ and any root $\alpha\in \Fqmn$ of $f$.}
\end{equation}
In particular, $N(\alpha, m,n;q)$ is unchanged if $\alpha$ is replaced by any of its conjugates with respect to the field extension $\Fqmn/\Fq$.

The relation between splitting subspaces of $\Fqmn$ of dimension $m$ and ordered bases of the form $\bavm$ should be quite clear by now. For ease of reference, this is stated below and we remark that this is just a paraphrasing of \cite[Lem. 3]{N2}. 

\begin{lemma}
\label{splitandbases}
Let $\alpha\in \Fqmn$ be such that $\Fqmn = \Fq(\alpha)$, and let $v_1, \dots , v_m\in \Fqmn$. Also let $W$ denote the $\Fq$-linear subspace 
of $\Fqmn$ spanned by $v_1, \dots , v_m$. Then $\bavm$ is an ordered basis of $\Fqmn$ if and only if $W$ is an $m$-dimensional splitting subspace of $\Fqmn$. 
\end{lemma}

\begin{proof}
Straightforward. 
\end{proof}

\begin{corollary}
\label{SandN}
Let $\alpha\in \Fqmn$ be such that $\Fqmn = \Fq(\alpha)$. Then 
$$
S(\alpha, m,n;q) = \frac{N(\alpha, m,n;q)}{\left|\GL_m(\Fq)\right|}, \quad \text{that is,} \quad N(\alpha, m,n;q) = S(\alpha, m,n;q) \prod_{i=0}^{m-1}(q^m-q^i).
$$
\end{corollary}

\begin{proof}
Follows from Lemma \ref{splitandbases} and the fact that the number of distinct ordered bases of an $m$-dimensional vector space over $\Fq$ is 
$\left|\GL_m(\Fq)\right| = \prod_{i=0}^{m-1}(q^m-q^i)$. 
\end{proof}

In view of \eqref{fiberandN} and Corollary \ref{SandN}, we can formulate the following quantitative formulation of (a slightly more general version of) Niederreiter's problem. 

\begin{conjecture}[Splitting Subspace Conjecture]
\label{SSC}
Let $\alpha\in \Fqmn$ be such that $\Fqmn = \Fq(\alpha)$. Then 
$$
S(\alpha, m,n;q) = \displaystyle \frac{q^{mn}-1}{q^m-1} \, q^{m(m-1)(n-1)} .
$$
\end{conjecture}

The above discussion makes it clear that Irreducible Fiber Conjecture (\ref{FibConj2}) 
and the Splitting Subspace Conjecture (\ref{SSC}) are equivalent to each other. In particular, Theorem \ref{mtwoconj} implies that the Splitting Subspace Conjecture holds in the affirmative when $m=2$. It may also be noted that the Splitting Subspace Conjecture is trivially valid when either $m=1$ or $n=1$, and thus this equivalent formulation of a more general version of the Primitive Fiber Conjecture (\ref{FibConj1}) subsumes \cite[Thm. 7.1]{GSM}.

In the remainder of this section, we will use some elementary observations to formulate a refined version of the Splitting Subspace Conjecture that seems particularly amenable to tackle. Let us first make some definitions. For $\alpha\in \Fqmn$, let $\mathfrak{S}_{\alpha}$ denote the set of 
all $m$-dimensional $\alpha$-splitting subspaces of $\Fqmn$. 
By a \emph{pointed $\alpha$-splitting subspace} of dimension $m$ we shall mean a pair $(W,x)$ where $W \in \mathfrak{S}_{\alpha}$ 
and $x\in W$. The element $x$ may be referred to as the \emph{base point} of $(W,x)$. Given any $x\in \Fqmn$, we let
$
\mathfrak{S}_{\alpha}^x:=\left\{W\in \mathfrak{S}_{\alpha}: x\in W\right\}.
$

\begin{proposition}
\label{elemsplitprop}
Let $\alpha\in \Fqmn$ be such that $\Fqmn = \Fq(\alpha)$. Then: 
\begin{enumerate}
	\item[{\rm (i)}] $\mathfrak{S}_{\alpha}$ is nonempty. Also, if $W\in \mathfrak{S}_{\alpha}$ and $\beta\in \Fqmn^*$, then $\beta W\in \mathfrak{S}_{\alpha}$.
\item[{\rm (ii)}]
$\mathfrak{S}_{\alpha}^x$ is nonempty for any $x\in \Fqmn^*$.
	\item[{\rm (iii)}]
$\left|\mathfrak{S}_{\alpha}^x\right| =\left|\mathfrak{S}_{\alpha}^y\right|$ for any $x,y\in \Fqmn^*$.
\item[{\rm (iv)}]
$\left|\mathfrak{S}_{\alpha}\right| = \left|\mathfrak{S}_{\alpha}^x\right|{(q^{mn}-1)}/{(q^m-1)}$ for any $x\in \Fqmn^*$.
\end{enumerate}
\end{proposition}

\begin{proof}
(i) Let 
$U$ be the $\Fq$-linear span of $\left\{\alpha^{in}: 0\le i\le m-1\right\}$. 
Then  $U\in \mathfrak{S}_{\alpha}$. Also, if $W\in \mathfrak{S}_{\alpha}$ and $\beta\in \Fqmn^*$, then $\beta W\in \mathfrak{S}_{\alpha}$ since
$\alpha^j\beta = \beta\alpha^j$ for $0\le j\le n-1$. 

(ii) 
If $U$ is as in (i), then 
$xU\in \mathfrak{S}_{\alpha}^x$ for any $x\in \Fqmn^*$. 

(iii) If $x,y\in \Fqmn^*$ and $\beta=y/x$, then $W\mapsto \beta W$ gives a bijection of $ \mathfrak{S}_{\alpha}^x$ onto $ \mathfrak{S}_{\alpha}^y$.

(iv) Counting the set $\left\{(W,x): W\in  \mathfrak{S}_{\alpha} \text{ and } x\in W\text{ with $x\ne 0$}\right\}$ 
of all pointed $\alpha$-splitting subspaces with a nonzero base point in two different ways, we 
find $\left|\mathfrak{S}_{\alpha}\right| (q^m-1) = \left|\mathfrak{S}_{\alpha}^x\right|{(q^{mn}-1)}$ for any $x\in \Fqmn^*$. 
\end{proof}

In view of parts (iii) and (iv) of Proposition \ref{elemsplitprop}, we can formulate the following refined version of the Splitting Subspace Conjecture.  

\begin{conjecture}[Pointed Splitting Subspace Conjecture]
\label{PSSC}
Let $\alpha\in \Fqmn$ be such that $\Fqmn = \Fq(\alpha)$ and let $x\in \Fqmn^*$. Then the number of $m$-dimensional pointed $\alpha$-splitting subspaces  of $\Fqmn$ with base point $x$ is equal to $q^{m(m-1)(n-1)}$. 
\end{conjecture}

It should be clear that the Pointed Splitting Subspace Conjecture implies all of the conjectures stated earlier, and also that the former is completely trivial when either $m=1$ or $n=1$. It may also be noted that part (i) of Proposition \ref{elemsplitprop} implies Proposition
\ref{surjtheta}. Finally, we remark that $q^{m(m-1)}$ is the number of nilpotent $m\times m$ matrices over $\Fq$, thanks to an old result of Fine and Herstein \cite{FH}, and thus a particularly nice way to prove 
the Pointed Splitting Subspace Conjecture could be to set up a natural bijection between $\mathfrak{S}_{\alpha}^x$ and the set of $(n-1)$-tuples (or if one prefers, pointed $n$-tuples) of nilpotent $m\times m$ matrices over $\Fq$.

\section{Asymptotic Formula} 
\label{sec:asymp}

The Irreducible Fiber Conjecture (\ref{FibConj2}) states that for any $f\in \imn$, the cardinality of $\Theta^{-1}(f)$ is  
$q^{m(m-1)(n-1)}  \prod_{i=1}^{m-1}(q^m-q^i)$. This expression is clearly a polynomial in $q$ of degree $mn(m-1)$. Even though the conjecture remains open, in general, we will show that asymptotically the size of each irreducible fiber is like $q^{mn(m-1)}$. 
To this end, we use \eqref{fiberandN}, and obtain suitable lower and upper bounds for 
$N(\alpha, m,n;q)$ by adapting an argument in the proof of \cite[Thm. 5]{N2}. 

\begin{lemma}
\label{boundsforN}
Let $\alpha\in \Fqmn$ be such that $\Fqmn = \Fq(\alpha)$. Then 
$$
 \frac{(q-2)q^{mn}+1}{(q-1)} \, q^{mn(m-1)} \le N(\alpha, m,n;q) \le \prod_{i=0}^{m-1}(q^{mn}-q^i) .
$$
\end{lemma}

\begin{proof}
Let us write
$$
{\mathfrak{V}} = \left\{(v_1, \dots , v_m)\in \Fqmn^m : \bavm \text{ is an ordered $\Fq$-basis of } \Fqmn\right\}.
$$
Evidently, if $(v_1, \dots , v_m)\in {\mathfrak{V}}$, then $v_1, \dots , v_m$ are linearly independent. Hence 
$$
N(\alpha, m,n;q) = \left|{\mathfrak{V}} \right| \le \prod_{i=0}^{m-1}(q^{mn}-q^i).
$$
On the other hand, if $(v_1, \dots , v_m)\in \Fqmn^m \setminus {\mathfrak{V}}$, then there is a nonzero $mn$-tuple 
$$
\mathbf{c} = \left(c_{11}, \dots , c_{1n}, \dots , c_{m1}, \dots , c_{mn}\right) \in \Fq^{mn}
\; \text{ such that } \; \sum_{i=1}^m\sum_{j=1}^n c_{ij} v_i \alpha^{j-1} =0.
$$
In other words, $(v_1, \dots , v_m)$ is in the kernel 
of the linear map 
$\phi_{\mathbf{c}} : \Fqmn^m \to \Fqmn$ given by 
$$
\phi_{\mathbf{c}} (u_1, \dots , u_m) := \gamma_1u_1 + \dots + \gamma_m u_m, \quad \text{ where }\quad
\gamma_i:= \sum_{j=1}^n c_{ij} \alpha^{j-1} \text{ for $1\le i \le m$.}
$$
It is clear that if $\mathbf{c}$ is replaced by a proportional tuple $\lambda \mathbf{c}$, where $\lambda\in \Fq^*$, then $\ker \phi_{\mathbf{c}} = \ker \phi_{\lambda\mathbf{c}}$. Moreover, since $\mathbf{c}\ne \mathbf{0}$ and $\alpha$ is of degree $\ge n$ over $\Fq$,  not all $\gamma_1, \dots , \gamma_m$ are  zero, and therefore by the Rank-Nullity Theorem, $\ker \phi_{\mathbf{c}}$ is of dimension $m-1$ over $\Fqmn$. It follows that
$$
 \Fqmn^m \setminus {\mathfrak{V}} \subseteq \bigcup_{\mathbf{c}\in \PP(\Fq^{mn})} \ker \phi_{\mathbf{c}} \quad \text{ and } \quad
 \left| \Fqmn^m \setminus {\mathfrak{V}}\right| \le \frac{q^{mn} -1}{q-1} \, q^{mn(m-1)}.
$$
Consequently, 
$$
N(\alpha, m,n;q) \ge \left(q^{mn}\right)^m - \frac{q^{mn} -1}{q-1} \, q^{mn(m-1)} = \frac{(q-2)q^{mn}+1}{(q-1)} \, q^{mn(m-1)}.
$$
This completes the proof.
\end{proof}

%

\begin{theorem}
\label{asymptformula}
For any $f\in \imn$, the fiber cardinality $\left|\Theta^{-1}(f)\right|$ is asymptotically equivalent to $q^{mn(m-1)}$ as $q\to \infty$.
\end{theorem}

\begin{proof}
Let $f\in \imn$ and let $\alpha\in \Fqmn$ be a root of $f$. 
From \eqref{fiberandN} and Lemma \ref{boundsforN}, we see that $L(q) \le \left|\Theta^{-1}(f)\right| \le U(q)$, where
$$
L(q):=   \frac{(q-2)q^{mn}+1}{(q-1)(q^{mn}-1)} \, q^{mn(m-1)} \quad \text{and} \quad U(q):= \prod_{i=1}^{m-1}(q^{mn}-q^i) .
$$
Further if we let $L^*(q):= \big((q-2)q^{mn} + 1 \big)q^{mn(m-2)-1}$, then $L^*(q)\le L(q)$ for $q>2$. Since both $L^*(q)$ and 
$U(q)$ are monic polynomials in $q$ of degree $mn(m-1)$, we obtain the desired result.
\end{proof}

It is clear that if $\alpha\in \Fqmn$ is such that $\Fqmn = \Fq(\alpha)$, then similar asymptotic formulae can be easily 
obtained for $N(\alpha, m,n;q)$ and $S(\alpha, m,n;q)$. 

\section{Application to Toeplitz matrices} 
\label{sec:toep}

Recall that a square matrix $A=\left(a_{ij}\right)$ is said to be a \emph{Toeplitz matrix} if $a_{ij} = a_{rs}$ whenever $i-j=r-s$. Thus every $n\times n$ Toeplitz matrix looks like
\begin{equation}
\label{top}
T_{\mathbf{c}} = \left(c_{n+i-j}\right)=\begin {pmatrix} 
c_n & \dots & c_2 & c_1 \\
c_{n+1} & \ddots &  & c_2 \\
\vdots & \ddots &  \ddots & \vdots \\
c_{2n-1} & \dots & c_{n+1} & c_n
\end{pmatrix} \quad \text{where} \quad 
\mathbf{c}= \left(c_1, c_2, \dots, c_{2n-1}\right).
\end{equation}
We denote by $\T_n(\Fq)$ the set of all Toeplitz matrices with entries in $\Fq$ and let $\TGL_n(\Fq):= \T_n(\Fq)\cap \GL_n(\Fq)$. 
It is clear that $\left|\T_n(\Fq)\right| = q^{2n-1}$. Determining $\left|\TGL_n(\Fq)\right|$ is far less obvious, but it is also given by a nice formula, namely,  
\begin{equation}
\label{TGLn}
\left|\TGL_n(\Fq)\right|  = q^{2n-1} - q^{2n-2} = q^{2n-1}\left(1 - \frac1q\right).
\end{equation}
A fairly involved proof of \eqref{TGLn} has recently been given by Kaltofen and Lobo \cite{KL} who also point out that Toeplitz matrices and the corresponding systems of equations are of much recent interest in symbolic computation. In fact, Toeplitz matrices are essentially equivalent to Hankel matrices and in this setting, \eqref{TGLn} was proved much earlier by Daykin \cite{D}. Here we will relate the determination of 
$\left|\TGL_n(\Fq)\right|$ to the results of Section \ref{sec:mtwo} and the existence of an irreducible trinomial (or binomial). 

\begin{proposition}
\label{TGLsize}
Let $q$ and $n$ be such that there exists an irreducible polynomial in $\Fq[X]$ of the form $X^{2n}-aX-b$, where $a,b\in \Fq$. Then 
$\left|\TGL_n(\Fq)\right|  = q^{2n-1} - q^{2n-2}$. 
\end{proposition}

\begin{proof}
Let $f= X^{2n}-aX-b$ be an irreducible polynomial in $\Fq[X]$ and let $\alpha$ be a root of $f$ in $\Fqtn$. Given any $\beta\in \Fqtn$, there are unique $c_0, c_1, \dots , c_{2n-1}\in \Fq$ such that $\beta = c_0 + c_1\alpha + \dots + c_{2n-1}\alpha^{2n-1}$. Now $\alpha^{2n} = a\alpha +b$ and   therefore 
$\alpha^{2n-1+s} = a\alpha^{s} +b\alpha^{s-1}$ for $1\le s \le n-1$. This implies that in the unique expression for $\beta \alpha^{j-1}$ as an  $\Fq$-linear combination of $1, \alpha, \dots , \alpha^{2n-1}$, the coefficient of $\alpha^{n+i-1}$ is $c_{n+i-j}$ for $1\le i,j\le n$. In other words, the matrix whose columns represent the coordinates of $1, \, \alpha , \, \dots \,  \alpha^{n-1},  \, \beta, \alpha \beta, \, \dots , \, \alpha^{n-1} \beta$ with respect to the ordered basis $\left\{1, \alpha, \dots , \alpha^{2n-1}\right\}$ is a $2n \times 2n$ block matrix of the form 
$$
\begin {pmatrix} 
I_n & B \\
\mathbf{0} & T_{\mathbf{c}}
\end{pmatrix}, 
$$
where $B\in \M_n(\Fq)$ 
and $T_{\mathbf{c}}$ is the Toeplitz matrix 
as in \eqref{top} above. It follows that $\Sb =  \left\{1, \, \beta, \, \alpha , \, \alpha \beta, \, \dots , \,  \alpha^{n-1} ,\, \alpha^{n-1} \beta\right\}$ is an ordered $\Fq$-basis of $\Fqtn$ if and only if the Toeplitz matrix $T_{\mathbf{c}}$ is nonsingular. Moreover, if 
$\mathbf{c} =  \left(c_1, c_2, \dots, c_{2n-1}\right)\in \Fq^{2n-1}$
is such that $T_{\mathbf{c}}$ is nonsingular, then there are exactly $q$ values of $\beta = c_0 + c_1\alpha + \dots + c_{2n-1}\alpha^{2n-1}$ (corresponding to different choices for $c_0$) such that $\Sb$ is an ordered $\Fq$-basis of $\Fqtn$. But we have seen in the proof of Lemma \ref{nobases} that the number of $\beta \in \Fqtn$ for which $\Sb$ is an $\Fq$-basis of $\Fqtn$ is $q^{2n-1}(q-1)$. Consequently, 
$\left|\TGL_n(\Fq)\right| = q^{2n-1}(q-1)/q$, as desired. 
\end{proof}

The question as to whether for every prime power $q$ and positive integer $d$, there is an irreducible trinomial in $\Fq[X]$ of degree $d$ appears to be rather delicate. For example, Swan \cite{Sw} showed that if $d$ is a multiple of $8$, then there are no irreducible trinomials over $\FF_2$ of degree $d$. We refer to the papers of von zur Gathen \cite{vzG} and Hanson, Panario and Thomson \cite{HPT} for the current state of art on this topic. 
At any rate, a trinomial (that can possibly be a binomial) meeting the hypothesis of Proposition \ref{TGLsize} does exist in many cases. 
To illustrate some of these, 
we will simply use the following classical result. 

\begin{proposition}[{\cite[Thm. 3.75]{LN}}]
\label{binomials}
Let $d$  be a positive integer $\ge 2$ and $b\in \Fq$ be such that $b\ne 0$. Also let $e$ be the order of $b$ in $\Fq^*$. Then 
$X^d-b$ is irreducible in $\Fq[X]$ if and only if each prime factor of $d$ divides $e$ but not $(q-1)/e$, and moreover $q\equiv 1 (\text{mod } 4)$ whenever $d\equiv 0 (\text{mod } 4)$.
\end{proposition}

\begin{corollary}
\label{Fermatprime}
Assume that $q$ is a power of an odd prime that is not a Fermat prime. Then there are infinitely many positive integers $n$ such that $X^{2n} -b$ is irreducible in $\Fq[X]$ for some $b\in \Fq$.
\end{corollary}

\begin{proof}
The assumption on $q$ implies that $q-1 = 2^r s$ for some integers $r,s$ such that $r\ge 1$, $s>1$, and $s$ is odd. Now let $\ell$ be a prime factor of $s$ and $n= \ell^i$ be any power of $\ell$, where $i\ge 1$. Also let $b$ be a primitive element of $\Fq^*$. Then $X^{2n}-b$ satisfies the hypothesis of Proposition \ref{binomials}. 
\end{proof}

We remark that some of the ideas in this section have eventually led to nice new proofs of \eqref{TGLn} in the general case; for details, we refer to \cite{GGR}. 
 
\section*{Acknowledgments}
  
We are grateful to Sartaj Ul Hasan for his careful reading of a preliminary version of this paper and some helpful suggestions.

\end{document}